\documentclass[12pt]{amsart}
\usepackage{amssymb}
\usepackage{epsfig}  		
\usepackage{epic,eepic}       
\usepackage{amsmath}
\usepackage{amsmath,amsfonts,amssymb}
\usepackage{graphicx}
\usepackage{amscd}
\usepackage{pb-diagram}
\usepackage{url}
\newtheorem{theorem}{Theorem}[section]

\usepackage[usenames,dvipsnames]{xcolor}
\usepackage{color}

\theoremstyle{definition}

\theoremstyle{remark}

\usepackage[noend]{algorithmic}
\usepackage{algorithm}
\floatname{algorithm}{Algorithm}

\begin{document}

\title{Simplification of complexes for persistent homology computations.}

\author{Pawe{\l} D{\l}otko}              
\email{pawel.dlotko@uj.edu.pl}
\address{Department of Mathematics, University of Pennsylvania, Philadelphia, PA 19104-6395, USA and Institute of Computer Science,
         Jagiellonian University,
         Lojasiewicza 6,
         Krakow, Poland}
\thanks{The first author was supported by DARPA grant FA9550-12-1-0416. The second author is supported by Foundation for Polish Science IPP Programme ''Geometry and Topology in Physical Models''.}

\author{Hubert Wagner}
\email{hubert.wagner@uj.edu.pl}
\address{Institute of Computer Science,
         Jagiellonian University,
         Lojasiewicza 6,
         Krakow, Poland}



\begin{abstract}
In this paper we focus on preprocessing for persistent homology computations. We adapt some techniques which were successfully used for standard homology computations. The main idea is to reduce the complex prior to generating its boundary matrix, which is costly to store and process. We discuss the following reduction methods: elementary collapses, coreductions (as defined by Mrozek and Batko) and acyclic subspace method (introduced by Mrozek, Pilarczyk and \.Zelazna). 
\end{abstract}


\maketitle

\section{Introduction}
Persistent homology starts being applied to a wide range of different practical problems, ranging from sensor networks to root architecture analysis. Performance, however, still tends to be a problem. In the following paper efficient preprocessing algorithms are proposed. In the case of standard (i.e. non-persistent) homology the following preprocessing techniques were successfully used: elementary collapses~\cite{whitehead}, coreductions~\cite{core} and the acyclic subspace method~\cite{acc}. The basic idea behind these techniques is to remove a number of cells, without affecting the homology (or affecting it in a controlled way). Importantly, these methods work on raw data, that is before the boundary matrix is produced. 

On a practical note, the memory overhead of storing the boundary matrix is significant. For example, a 3-dimensional image of size $2000^3$ voxels with 16b gray-scale values takes roughly 16GB, while its boundary matrix roughly 200GB. To handle data of such sizes, it is crucial to pre-process before generating the boundary matrix.

The goal of this paper is to extend the mentioned preprocessing methods (elementary collapses, coreductions and acyclic subspace) to \emph{persistent} homology.  The techniques presented in this paper are heuristics, meaning there are no provable bounds on how many cells are reduced.  They have however performed well in practical situations as described in~\cite{MMBook}, \cite{core} and~\cite{acc}. The computational complexity of the presented techniques is linear in the number of cells provided the number of neighbors of every cell in the complex is $O(1)$ (this assumption is true for cubical and simplicial complexes which are used in practice). Therefore they can be used as an inexpensive preprocessing step. 

\section{Background}
\label{sec:background}
\subsection{Chain complexes and homology}
The presented methods work for arbitrary \emph{chain complexes} with field coefficients. In the most typical case this chain complex comes from a CW-decomposition of a given space. In practice, simplicial and cubical complexes are used. For simplicity, we use $\mathbb{Z}_2$ coefficients throughout the paper, as this is the standard setting for persistence. However, the presented algorithms works for any field coefficients. Intuitively, in this setting, a chain complex can be viewed as a set of abstract cells (e.g. cubes or simplices) connected by a boundary operator as specified below.

Let us fix a chain complex $\mathcal{K}$. Let a \emph{$p$-chain} be a formal sum of $p$-cells with the $\mathbb{Z}_2$ coefficients. The $p$-chains of $\mathcal{K}$, together with addition modulo 2, form a \emph{group of $p$-chains}, denoted by $C_p(\mathcal{K})$. The boundary operator $\partial_p$ maps $p$-chains into $(p-1)$-chains, called \emph{faces}. The definition of chain complexes requires that $\partial_p \circ \partial_{p+1} = 0$.

The chain of (co-)faces is called a (co-)boundary. For any $p$-chain $c = \sum a_{i} c_{i}$, we have $\partial_{p}c = \sum a_i \partial_{p}c_{i}$. If a $(p-1)$-cell $a$ has a $p$-cell $B$ in its coboundary, we say $a$ is a face of $B$, and $B$ is a coface of $a$. (Notation: capital letters denote higher dimensional cell where a cell and its face is considered). 
Let us have a set $A \subset \mathcal{K}$. By the boundary of $A$ we mean $bd\ A = \{ b \in \mathcal{K} | \exists_{a \in \mathcal{A}} b \in \partial a \}$. We say that $A$ is closed if $bd\ A \subset A$.


To define homology let us first introduce the group of $p$-cycles, $Z_p(\mathcal{K}) = ker \partial_p $ \and its subgroup: the group of $p$-boundaries, $B_p(\mathcal{K}) = im \partial_{p+1}$. The $p$-th homology group is the quotient $H_p(\mathcal{K}) = Z_p(\mathcal{K})/B_p(\mathcal{K})$. The $p$-th Betti number, denoted by $\beta_p(\mathcal{K})$, is the rank of this group.

\subsection{Filtrations and persistence}
\label{sec:introToPersistence}
Given a complex $\mathcal{K}$ and a filtering function $g:\mathcal{K}\rightarrow \mathbb{Z}$, \emph{persistent} homology studies homological changes of the sub-level complexes, $\mathcal{K}_t=g^{-1}(-\infty,t]$.  We require that $g(a) \leq g(B)$ whenever $a$ is a face of $B$ which implies that for every $t \in \mathbb{Z}$, $\mathcal{K}_t$ forms a complex i.e. the boundary of every cell in $\mathcal{K}_t$ is contained in $\mathcal{K}_t$ (we call this \emph{sublevel-complex filtration}). Although in general the filtration has values in $\mathbb{Z}$ sometimes to simplify the notation we assume that filtration starts from zero. Every filtration on a finite cell complex can be transformed to this one by a monotone change of coordinates, therefore this extra assumption does not limit the generality of the presented approach.

We say that $g$ is \emph{filtration by maxima} if for every $A \in \mathcal{K}$, $g(A) = max\{ g(b_1),\ldots,g(b_n)\}$ where $b_1,\ldots,b_n$ are the faces of $A$. Such a filtration is often used when the function values are given only on the vertices and have to be extended to higher-dimensional cells. All the presented methods works for general filtrations, but this type of filtration is convenient for examples.

Persistent homology captures the birth and death times of homology classes of the sub-level complexes, as $t$ grows from $-\infty$ to $+\infty$.
By birth, we mean that a homology class is created; by death, we mean it either becomes trivial or becomes identical to some other class born earlier. 
The \emph{persistence}, or lifetime of a class, is the difference between the death and birth times. Often a multiset of \emph{persistence intervals} is used to represent persistence in a given dimension. A single interval encodes the lifetime of a homology class. We say that two spaces have same persistence, if their persistence intervals are the same in the corresponding dimensions. We disregard zero-length persistence intervals, because they carry virtually no information. 

The formal definition is as follows (after~\cite{herbert}): The $p$-th persistent homology groups of filtered complex $\mathcal{K}$ are the images of the homomorphisms induced by inclusion, $H^{i,j}(\mathcal{K}) = im H(f^{i,j})$, where $f^{i,j} : \mathcal{K}_i \hookrightarrow \mathcal{K}_j$ for every $i,j \in \mathbb{Z}$, $i \leq j$. By $H^p(\mathcal{K})$ we denote the persistence homology of a filtered cell complex $\mathcal{K}$. So-called \emph{persistence diagrams} encode the entire information about the persistence of a filtered complex. For more details see~\cite{herbert}.

We want to remind a theorem saying when the persistence of two filtered complexes are equal:
\begin{theorem}[Persistence equivalence theorem,~\cite{herbert}]
\label{thm:pet}
Consider persistent homology of two filtered complexes $X$ and $Y$. Let $\phi_i : H_{*}(X_i) \rightarrow H_{*}(Y_i)$:
$$
\begin{CD}
  \ldots H_{*}(X_0) @>>> H_{*}(X_1) @>>> \ldots @>>> H_{*}(X_{n-1}) @>>> H_{*}(X_{n}) \ldots \\
  @V\phi_0VV @V\phi_1VV @. @V\phi_{n-1}VV @V\phi_nVV\\
  \ldots H_{*}(Y_0) @>>> H_{*}(Y_1) @>>> \ldots @>>> H_{*}(Y_{n-1}) @>>> H_{*}(Y_n) \ldots\\
\end{CD}
$$

If the $\phi_i$ are isomorphisms and all the squares commute, then the persistent homology of $X$ and $Y$ is the same.
\end{theorem} 

The standard algorithm to compute persistence homology is a matrix reduction algorithm presented in~\cite{herbert}. The other algorithms to compute persistence are discussed in the Section~\ref{sec:algAndComplexity}. The output of the persistent homology computations is a list of \emph{persistence pairs} of the form (birth, death).

An important justification for the usage of persistence is the stability theorem. Cohen-Steiner et al.~\cite{stability} proved that for any two filtering functions $g$ and $h$ the so called
 bottleneck distance ($d_B$, see~\cite{stability}) between the persistence of $\mathcal{K}$ with respect to $f$ (denoted as $H^p(\mathcal{K},f)$) and the persistence of $\mathcal{K}$ wrt. $g$ ($H^p(\mathcal{K},g)$) is upper bounded by the $L^\infty$ norm of the difference between $f$ and $g$:
 \begin{equation}
 d_B(H^p(\mathcal{K},f), H^p(\mathcal{K},g)) \leq \lVert f-g \rVert_\infty := \max_{x\in \mathcal{K}} \lvert f(x)-g(x) \rvert.
 \label{eq:stability}
 \end{equation}

Simply put, small changes in the filtration values cause small changes in persistence. This enables robust estimation of how persistence is affected by perturbation of the input (e.g. noise).

\subsection{Existing algorithms and their complexity}
\label{sec:algAndComplexity}
Applying the matrix-reduction algorithm described in~\cite{herbert} to the boundary matrix of the input complex is the standard way to compute persistent homology groups. It works for general complexes in arbitrary dimensions. The worst-case complexity is $O(n^3)$, where $n$ is the size of the input complex. Milosavljevic et al.~\cite{milosavljevic2011zigzag} showed that persistent homology can be computed in matrix multiplication time $O(n^\omega)$ where the currently best estimation of $\omega$ is $2.373$. Chen and Kerber~\cite{chen2011output} proposed a randomized algorithm to compute only pairs with persistence above a chosen threshold. Despite improving the theoretical complexity, it is unclear whether these methods are better in practice.

When focusing on $0$-dimensional homology, union-find data structures can be used to compute 0-dimensional persistence in time $O(n \alpha(n))$ \cite{herbert}, where $\alpha$ is the inverse of the Ackermann function and $n$ the size of the input. 

A recent variation of the standard algorithm, introduced by Chen and Kerber~\cite{chao} significantly reduces the amount of computations. This idea was also used by Wagner et al.~\cite{wagner11} to compute persistence for $n-$dimensional images. 

One can also compute persistence by computing homology of inclusion map as discussed in~\cite{MMTW} and~\cite{natalia}. This approach gives even reacher information than persistence. It is also very time consuming as noted in~\cite{gunther}.

Another class of methods involves preprocessing the input complex using discrete Morse theory, as proposed by Robins~\cite{vanessa}. Such a preprocessing significantly reduces the size of the boundary matrix, while preserving persistence. In the case of 3D grayscale images, an efficient parallel implementation was proposed in~\cite{gunther}, allowing for handling large ($\approx 1200^3$) images on commodity hardware. The standard matrix-reduction algorithm is used in the final step of computations.

The approach by Robins works for arbitrary complexes and in dimension three the preprocessing results in the smallest possible boundary matrix~\cite{vanessa} (counting the number of rows/columns). The algorithm used in~\cite{vanessa} depends crucially on simple-homotopy theory, which makes it hard to directly generalize the optimality result to higher dimensions. A recent paper by Mischaikow et al.~\cite{mischaikow} proposes a handy theoretical framework, where discrete Morse theory is extended from complexes to filtrations. 

It should be noted that the simplification methods based on discrete Morse theory aim at optimizing the number of cells. The total size of the resulting complex can in fact grow, since the number of connections can grow even quadratically in the number of cells. In other words, even if the initial boundary matrix is sparse, the matrix of the simplified complex can be dense. 
The methods presented in this paper do not have this property. Also, they can be used prior to persistence computations with discrete Morse theory.

\section{Elementary collapses.}
Elementary collapse is an old concept introduced by Whitehead in a context of simple homotopy types in~\cite{whitehead}. It is often used in the context of homology theory. It finds a pair of cells $(A,b) \in \mathcal{K} \times \mathcal{K}$ (called elementary collapse pair) such that $b$ has only one coface $A$ in $\mathcal{K}$ (in this case $b$ is called a \emph{free face}). 
Removing such a pair does not change homology or homotopy type of $\mathcal{K}$, since such a removal corresponds to a deformation retraction. 

In this section we show that the elementary collapses can be used also in the case of persistent homology. However, the elementary collapse pair $(A,b)$ must fulfill two extra assumptions:
\begin{enumerate}
\item $b$ is a free face for every $\mathcal{K}_n$, $n \in \mathbb{Z}$.
\item $g(A) = g(b)$.
\end{enumerate}
The first condition indicates that at every filtration level $(A,b)$ is a elementary collapse pair. It suffices to check the size of coboundary of $b$ in the final complex in the filtration. The second condition indicates that filtration levels of $A$ and $b$ need to be equal. In Figure~\ref{fig:elementaryRed} we show that these are in fact necessary conditions in order to preserve persistent homology information.
\begin{figure}[!h]
\centerline{\includegraphics[scale=0.3]{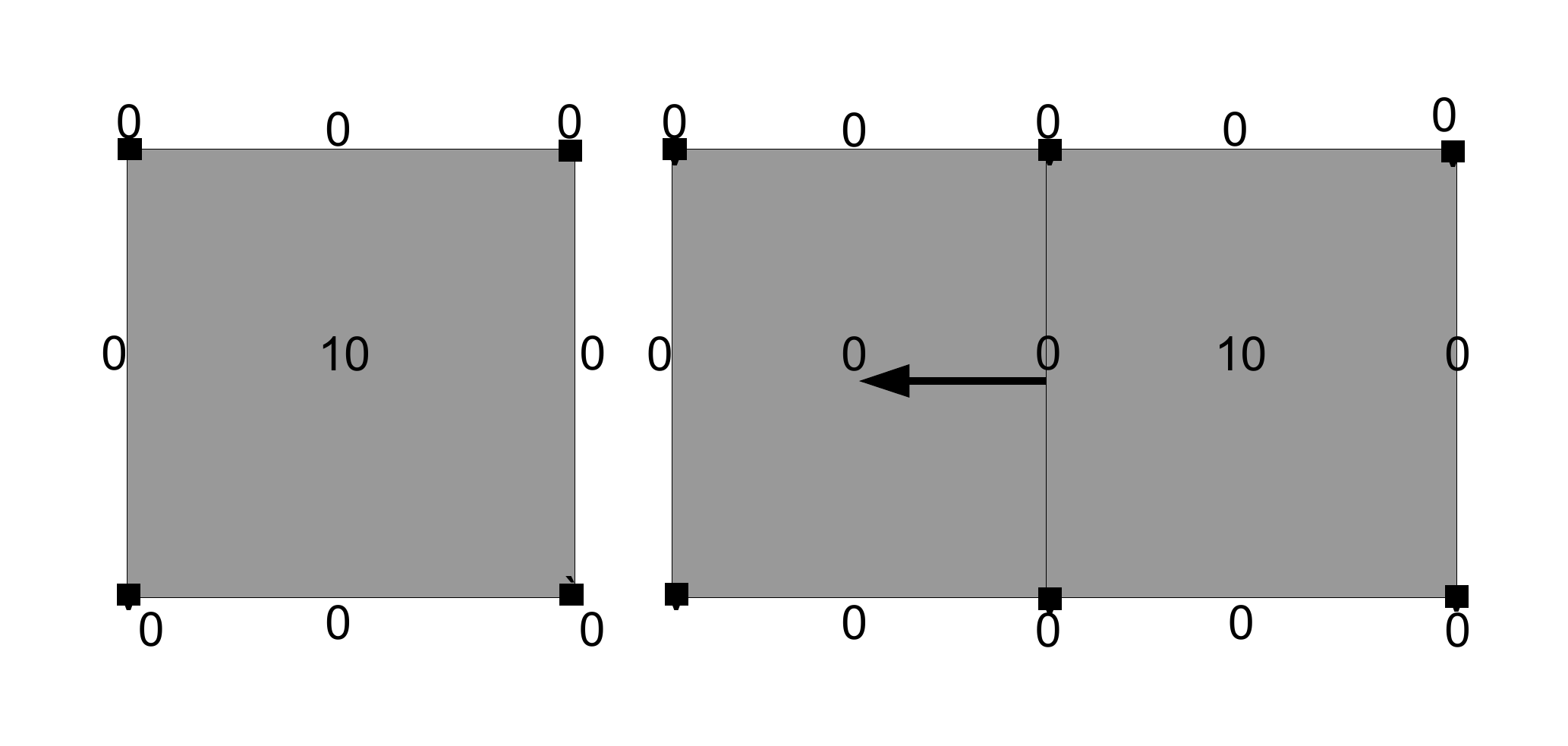}}
\caption{On the left the complex consisting of a single cube is depicted. All the vertices and edges have filtration level $0$, while the two dimensional cube has filtration level $10$. Consequently in the case of one-dimensional persistent homology an interval $[0,10]$ appears. However when any elementary collapse is performed, the interval does not appear anymore. Therefore the assumption $g(A) = g(b)$ is necessary. On the right picture, if an elementary collapse is performed at level $0$ as indicated by the arrow, at sub-level $10$ we are missing the shared edge and the 2-cell having value $0$, so the homology (and consequently persistence) is changed.
}
\label{fig:elementaryRed}
\end{figure}

Let us now prove that the presented extra assumptions guarantee that after removing a pair $(A,b)$, persistent homology of a complex $\mathcal{K}$ does not change. 

\begin{theorem}
\label{th:reductions}
Let $(A,b) \in \mathcal{K} \times \mathcal{K}$ be a elementary collapse pair and $b$ be a free face in $\mathcal{K}_n$ for every $n \in \mathbb{Z}$. Moreover let $g(A) = g(b)$. Then $H_{*}^p(\mathcal{K}) = H_{*}^p(\mathcal{K} \setminus \{A,b\})$.
\end{theorem}
\begin{proof}
The proof bases on Theorem~\ref{thm:pet}. Let us consider maps in homology induced by inclusions $\mathcal{K}_1 \subset \mathcal{K}_2 \subset \ldots \subset \mathcal{K}_n$ before and after removal of a pair $(A,b)$:
$$
\begin{CD}
  \ldots @>H(i_{l-1})>> H_{*}(\mathcal{K}_l) @>H(i_{l})>> H_{*}(\mathcal{K}_{l+1}) @>H(i_{l+1})>> \ldots \\
  @VVV@VV H(j_{l}) V@VV H(j_{l+1}) V @VVV\\
  \ldots @>H(k_{l-1})>> H_{*}(\mathcal{K}_l \setminus \{A,b\}) @>H(k_{l})>> H_{*}(\mathcal{K}_{l+1} \setminus \{A,b\}) @>H(k_{l+1})>> \ldots\\
\end{CD}
$$

The horizontal maps are the maps induced in homology by inclusion map $i_l: \mathcal{K}_{l} \hookrightarrow  \mathcal{K}_{l+1}$ and $k_l: \mathcal{K}_{l} \setminus \{A,b\} \hookrightarrow  \mathcal{K}_{l+1} \setminus \{A,b\}$. The vertical maps $j_l : H_{*}(\mathcal{K}_{l}) \rightarrow H_{*}(\mathcal{K}_{l}\setminus \{A,b\})$ are isomorphisms, since a pair $(A,b)$ is elementary collapse pair in every $\mathcal{K}_l$ and $g(A) = g(b)$. 

To show that all squares commute, let us pick a chain $c \in C(\mathcal{K}_l)$. Then $j_l(c) = c -  \langle c , A \rangle A - \langle c , b\rangle b$ i.e. the generators $A$ and $b$ are removed form the chain $c$. We see that $k_l( j_l(c))$ is simply $j_l(c) \in C(\mathcal{K}_{l+1} \setminus \{A,b\})$. Let us now compute $j_{l+1}(i_l(c))$. We have $i_l(c) = c \in C(\mathcal{K}_{l+1})$. To obtain $j_{l+1}(i_l(c))$ it again suffices to remove from $i_l(c)$ the generators $A$ and $b$, which is $j_{l+1}(i_l(c)) =  c -  \langle c , A \rangle A - \langle c , b\rangle b$. Therefore $j_{l+1}(i_l(c)) = k_l( j_l(c) )$ so all the squares commute.

Therefore due to Theorem~\ref{thm:pet} persistence modules defined by lower and upper sequence of complexes are equal, which completes the proof. 
\end{proof}

It might seem that there could be just a few elementary collapse pairs to remove in the external boundary of the considered complex. However the reduction process usually creates new collapse pairs and it can be continued.

The idea of the described procedure is summarized in Algorithm~\ref{alg:elemCollapses}.

\begin{algorithm}[!ht]
  \small
  \caption{Elementary collapses for persistence.}
  \label{alg:elemCollapses}
  \begin{algorithmic}
	\REQUIRE $\mathcal{K}$ - cell complex with filtration $g : \mathcal{K} \rightarrow \mathbb{Z}$;
	\ENSURE Reduced complex $\mathcal{K}'$ with the same persistence as $\mathcal{K}$.
	\STATE $\mathcal{K}' = \mathcal{K}$ 
    \STATE List of cells \texttt{Q}$= \emptyset$;
    \FOR {every cell $b \in \mathcal{K}'$}
		\IF {$b$ has unique cell $A$ in coboundary and $g(b) = g(A)$}
			\STATE \texttt{Q}$.enqueue(b)$;
		\ENDIF
    \ENDFOR
    \WHILE { \texttt{Q} is not empty }
		\STATE $b\ =$ \texttt{Q}$.dequeue()$;
		\IF { $b$ has unique cell $A$ in coboundary and $g(b) = g(A)$ }
			\STATE $\mathcal{K}' = \mathcal{K}' \setminus \{A,b\}$;
			\FOR {every element $c$ in boundary of $A$ or boundary of $b$}
				\IF {$c$ has unique cell $D$ in coboundary and $g(c) = g(D)$}
					\STATE \texttt{Q}$.enqueue(c)$;
				\ENDIF
			\ENDFOR
		\ENDIF
    \ENDWHILE
  \end{algorithmic}
\end{algorithm} 
\section{Coreductions.}
The concept of coreductions was introduced by Mrozek and Batko in~\cite{core}. It is based on an idea to search for an homologically trivial sub-complex in the given complex starting from lowest possible dimension (bottom-up). The approach to find acyclic subcomplexes that uses only top dimensional cells are shown in Section~\ref{sec:accSubspace}. The formal presentation requires the concept of S-complex, which is a chain complex with a fixed basis. This concept is not recalled here for brevity's sake. For further details consult~\cite{core}.

Pair of cells $(A,b)$ is a \emph{coreduction pair} if $b$ is the unique element in boundary of $A$. Such a pair cannot exist in a simplicial or cubical complex. Therefore in~\cite{core} first a single vertex is removed from each connected component of the considered complex. This removal changes only the zero dimensional homology. Now, the process of coreduction is iterated as long as one can find a coreduction pair in the considered complex. 

The coreduction algorithm was already used to compute homology of inclusions and persistent homology in~\cite{MMTW}. However the approach there was different -- in~\cite{MMTW} the coreduction was used to remove as many elements as possible before computing map in homology induced by inclusion map between levelsets. To be precise: at each level of filtration the homology generators were computed and then, by solving system  of linear equations one obtained the matrix mapping generators at $n$th level to generators at $(n+1)$th level. 
This procedure is more general than computing persistence. As indicated in~\cite{wagner11} this strategy is in general not efficient for computing only persistence. Here we use coreduction algorithm as a preprocessing for the standard algorithm to compute persistence, as presented in~\cite{herbert}.

First we show that removing a coreduction pair $(A,b)$ such that $g(A) = g(b)$ does not change the persistent homology. Then we show that removing a vertex changes only $0$-dimensional persistence.

Since Corollary 4.2. from~\cite{core} is extensively used in this section, we recall it here:
\begin{theorem}[Corollary 4.2~\cite{core}]
\label{th:42MMBB}
Let $\mathcal{K}$ be a chain complex (without filtration). If $(A,b)$ is a coreduction pair in $\mathcal{K}$, then $H(\mathcal{K})$ and $H(\mathcal{K} \setminus \{A,b\})$ are isomorphic.
\end{theorem}

Now we are ready to present the main theorem of this section, analogous to Threorem~\ref{th:reductions}:

\begin{theorem}
\label{th:coreductions}
Let $(A,b) \in \mathcal{K} \times \mathcal{K}$ be a coreduction pair. Moreover let $g(A) = g(b)$. Then $H_{*}^p(\mathcal{K}) = H_{*}^p(\mathcal{K} \setminus \{A,b\})$. 
\end{theorem}

\begin{proof}
First we point out that if $(A,b)$ is a coreduction pair and $g(A) = g(b)$, then it is a coreduction pair in every filtration level in which $A$ and $b$ exist. 
Let us consider the filtered complex $\mathcal{K}$ before and after removal of a coreduction pair $(A,b)$:

$$
\begin{CD}
  \ldots @>H(i_{l-1})>> H_{*}(\mathcal{K}_l) @>H(i_{l})>> H_{*}(\mathcal{K}_{l+1}) @>H(i_{l+1})>> \ldots \\
  @VVV@VV H(j_{l}) V@VV H(j_{l+1}) V @VVV\\
  \ldots @>H(k_{l-1})>> H_{*}(\mathcal{K}_l \setminus \{A,b\}) @>H(k_{l})>> H_{*}(\mathcal{K}_{l+1} \setminus \{A,b\}) @>H(k_{l+1})>> \ldots\\
\end{CD}
$$

The horizontal maps $i_l: \mathcal{K}_{l} \hookrightarrow  \mathcal{K}_{l+1}$ and $k_l: \mathcal{K}_{l} \setminus \{A,b\} \hookrightarrow  \mathcal{K}_{l+1} \setminus \{A,b\}$ are inclusion maps. The vertical maps $H(j_l) : H_{*}(\mathcal{K}_{l}) \rightarrow H_{*}(\mathcal{K}_{l}\setminus \{A,b\})$ are isomorphism due to the Theorem~\ref{th:42MMBB} and the fact that $\{A,b\}$ is a coreduction pair in both $\mathcal{K}_l$ and $\mathcal{K}_{l+1}$ and $g(A) = g(b)$. Analogously as in Theorem~\ref{th:reductions} one can show that all the squares commute. Therefore, due to Theorem~\ref{thm:pet} persistence defined by lower and upper sequence of complexes is equal, which completes the proof.

\end{proof}

It remains to show that we can remove a single vertex $V$ from the complex, and this changes only the zero dimensional persistent homology. The fact that higher dimensional persistence is not affected follows from the definition of reduced homology and persistence. To be precise -- let us remind the augmented chain complex in the setting of reduced homology:

\[
\xrightarrow{\partial_{n+1}} C_n(\mathcal{K}) \xrightarrow{\partial_{n}} \ldots \xrightarrow{\partial_{2}} C_1(\mathcal{K}) \xrightarrow{\partial_{1}} C_0(\mathcal{K}) \xrightarrow{\epsilon} \mathbb{Z}_2 \rightarrow 0
\]
 where $\epsilon(\sum_i \alpha_i t_i) = \sum_i \alpha_i$ for $t_i \in \mathcal{K}$.
 
 In this setting, removal of the initial vertex $V$ can be interpreted as a removal of a coreduction pair $(V,\emptyset)$, where $\emptyset$ is the unique generator of $\mathbb{Z}_2$ in dimension -1. From the properties of the reduced homology it follows that removing $V$ changes only the zero dimensional homology (and persistence).

In the standard homology, once the coreduction pairs are removed as long as possible, all vertices in a given connected component are removed. The example in Figure~\ref{fig:elementaryCoredNotSufficeForCC} show that in the case of persistent homology this does not hold: Even if we start from the vertex with the lowest filtration value and perform the coreductions as long as possible, some vertices may remain in the considered connected component.

\begin{figure}[!h]
\centerline
{\includegraphics[scale=0.3]{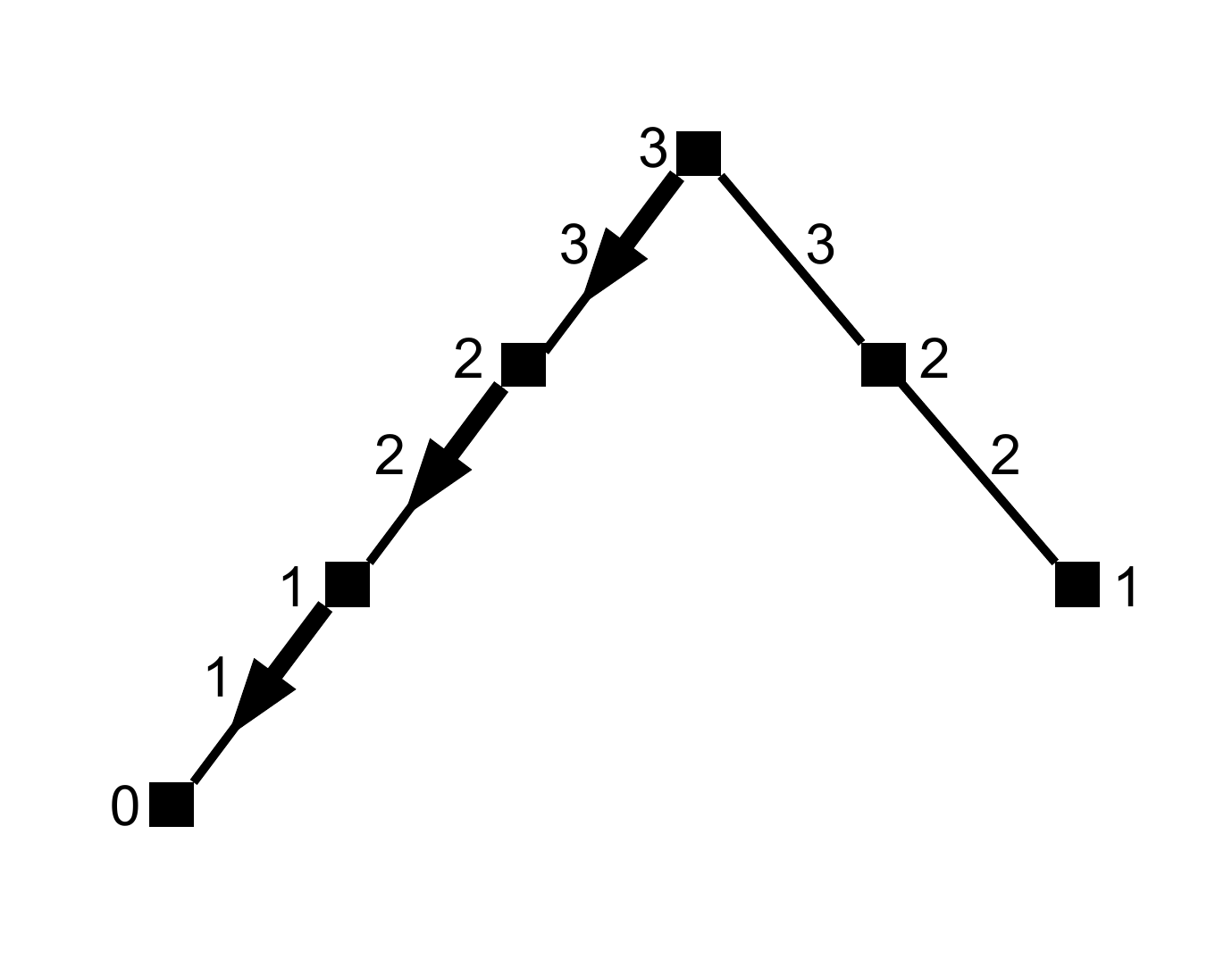}}
\caption{One dimensional complex, filtered with maxima. Numbers indicate the filtration values. Initially the vertex having value $0$ is removed. Arrows indicate a pairings done by coreductions. This example shows that unlike the case of standard homology, not all vertices form the connected component can be removed during coreductions.}
\label{fig:elementaryCoredNotSufficeForCC}
\end{figure}
As it can be seen at Figure~\ref{fig:zeroDimIsLost} the whole information about zero-dimensional persistence is lost when the coreductions are done. However it is fairly easy to compute zero dimensional persistence in near linear time from the original complex by using find-union data structures~\cite{herbert}.
\begin{figure}[!h]
\centerline{\includegraphics[scale=0.3]{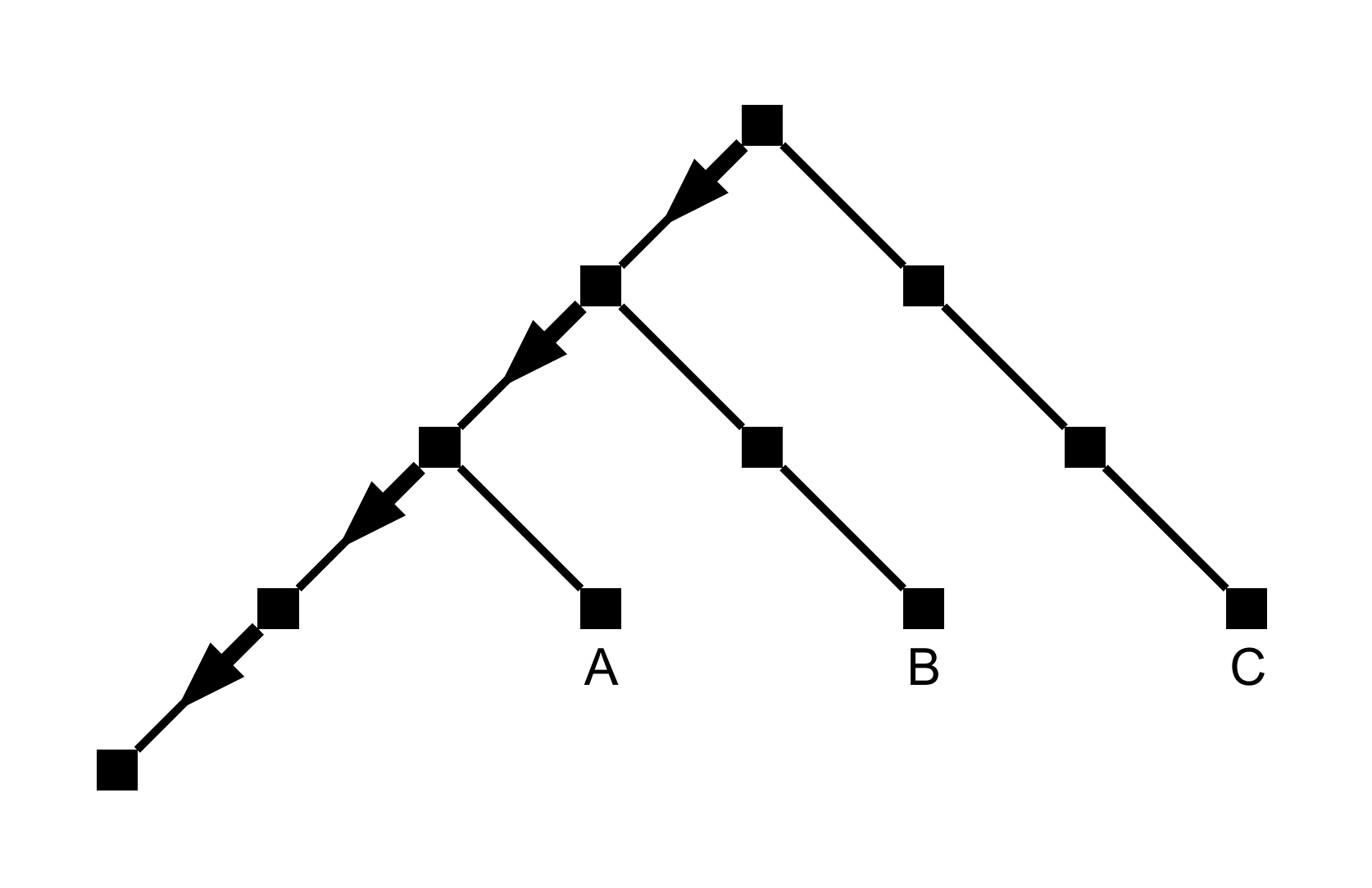}}
\caption{An example showing that when the coreductions are done the information about zero dimensional persistence is lost. For the sake of clarity the filtration is simply a height function (it is not depicted with numbers). We assume the filtration of an edge is the maximal filtration level of its vertices. The coreductions are started by removing left-bottom vertex and are indicated by arrows. It is easy to see that in the coreduced complex vertices $A$, $B$ and $C$ generates infinite persistence intervals in zero dimensional persistent homology. Those infinite intervals are not present in persistent homology of the initial complex.}
\label{fig:zeroDimIsLost}
\end{figure}

As it can be seen in Figure~\ref{fig:elementaryCored} removing a coreduction pair $(A,b)$ such that $g(A) > g(b)$ does change the persistent homology of the given complex. Therefore the  assumptions presented in Theorem~\ref{th:coreductions} are crucial.  
\begin{figure}[!h]
\centerline{\includegraphics[scale=0.3]{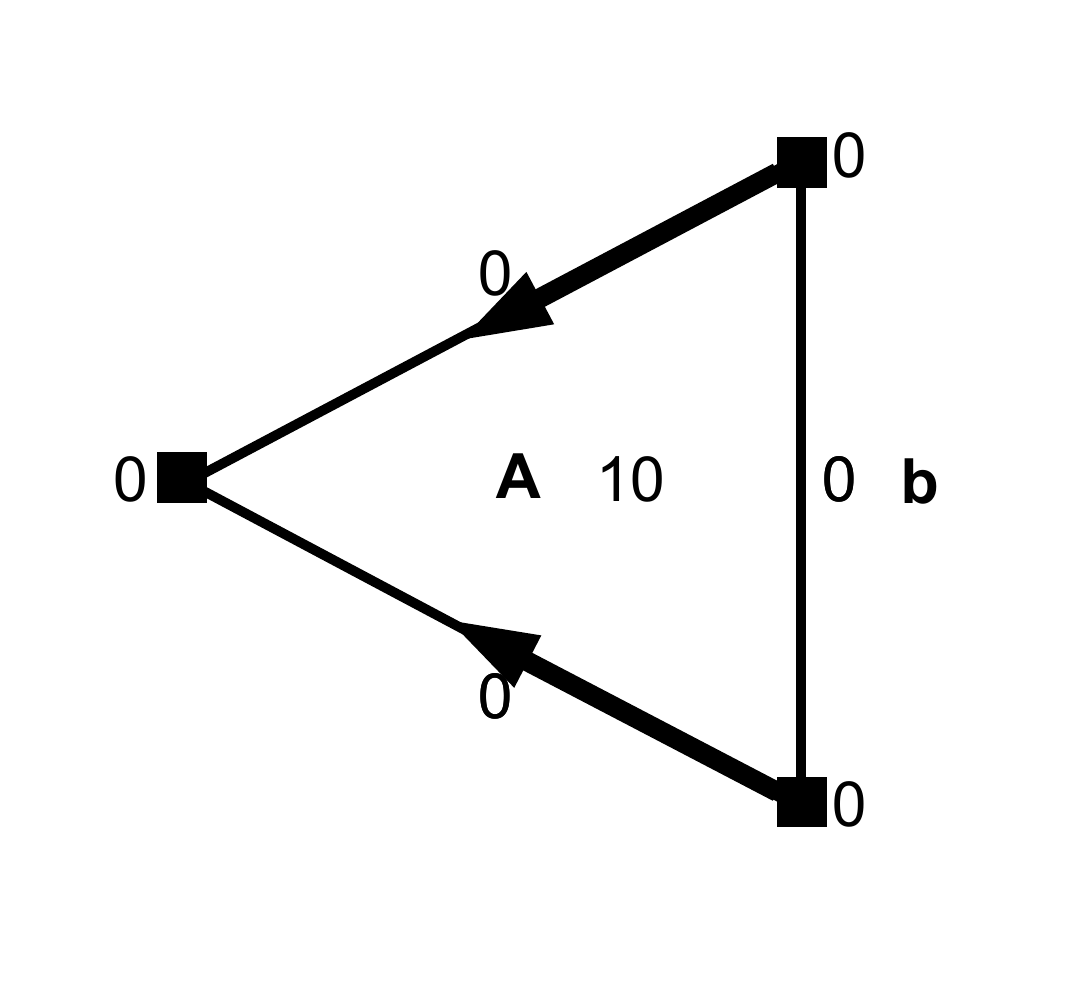}}
\caption{Simple demonstration showing that condition $g(A) = g(b)$ is necessary when removing coreduction pair $(A,b)$. Suppose a coreduction pair $(A,b)$ presented in the picture is removed (the vertex on the left and the edges paired with remaining vertices are already removed from the complex). It is then clear that once pair $(A,b)$ is removed an interval $[0,10]$ is lost from one dimensional persistence.}
\label{fig:elementaryCored}
\end{figure}

To summarize this section in Algorithm~\ref{alg:core} the coreductions for persistence are presented.

\begin{algorithm}[!ht]
  \small
  \caption{Coreductions for persistence.}
  \label{alg:core}
  \begin{algorithmic}
	\REQUIRE $\mathcal{K}$ - cell complex with filtration $g : \mathcal{K} \rightarrow \mathbb{Z}$;
	\ENSURE Reduced complex $\mathcal{K}'$ with the same persistence for dimensions $\geq 1$ as $\mathcal{K}$.
	\STATE $\mathcal{K}' = \mathcal{K}$; 
    \STATE List of cells \texttt{Q}$= \emptyset$;
    \FOR {every connected component of $\mathcal{K}'$}
		\STATE Remove a single vertex $V$ from considered connected component of $\mathcal{K}'$;
    \ENDFOR
    \FOR {every cell $B \in \mathcal{K}'$}
		\IF {$B$ has unique cell $a$ in boundary and $g(B) = g(a)$}
			\STATE \texttt{Q} $.enqueue(B)$;
		\ENDIF
    \ENDFOR
    \WHILE { \texttt{Q} is not empty }
		\STATE $B\ =\ $\texttt{Q}$.dequeue()$;
		\IF { $B$ has unique cell $a$ in boundary and $g(B) = g(a)$ }
			\STATE $\mathcal{K}' = \mathcal{K}' \setminus \{B,a\}$;
			\FOR {every element $C$ in coboundary of $a$ or coboundary of $B$}
				\IF {$C$ has unique cell $d$ in boundary and $g(C) = g(d)$}
					\STATE \texttt{Q} $.enqueue(C)$;
				\ENDIF
			\ENDFOR
		\ENDIF
    \ENDWHILE
  \end{algorithmic}
\end{algorithm} 
\section{Acyclic subspace.}
\label{sec:accSubspace}
The acyclic subspace method follows from the excision property. In standard homology theory it states that for a complex $\mathcal{K}$ and a subcomplex $\mathcal{A}$ of $\mathcal{K}$ such that $H_{n}(\mathcal{A}) = 0$ for $n > 0$ and $H_{0}(\mathcal{A}) = \mathbb{Z}^n$, where $n$ is the number of connected components of $\mathcal{K}$, we have $H_{n}(\mathcal{K}) = H_n(\mathcal{K},\mathcal{A})$ for $n>0$.
Since $\mathcal{A}$ is closed we can use the following theorem from~\cite{core}:
\begin{theorem}[Theorem 3.1~\cite{core}]
\label{th:MMBB31}
If $\mathcal{A}$ is closed in $\mathcal{K}$, then $H_{*}( \mathcal{K} \setminus \mathcal{A} ) = H_{*}( \mathcal{K} , \mathcal{A} )$.
\end{theorem}
Therefore it suffices to compute homology of a chain complex obtained by removing all elements in $\mathcal{A}$ from $\mathcal{K}$. If one is able to efficiently
 find large acyclic subcomplex $\mathcal{A}$, then this approach offers great performance gains. For instance this method has been used to speed up cubical homology
 computations in~\cite{acc} and to efficient computation of cohomology generators in~\cite{cmes}. A similar technique has been used to speed up computations of
 homology of inclusions~\cite{natalia}.

The presented strategy can be especially useful in the case of cubical data, for example 3D images. In this case we work only on top dimensional cells. This is important, as the number of faces (of all dimensions) of a given cell is exponential in its dimension. Even though this method is limited to low dimensions, the performance gains can be significant.

In this section we demonstrate how this technique can be used to speed up computations of persistent homology. To do this, we need to introduce the concept of \emph{acyclic subcomplex compatible with filtration}. Let $\mathcal{A}$ be a subcomplex of a filtered complex $\mathcal{K}_0 \subset \ldots \subset \mathcal{K}_n = \mathcal{K}$. We say that $\mathcal{A}$ is an \emph{acyclic subcomplex compatible with filtration} if $\mathcal{A}_i = \mathcal{K}_i \cap \mathcal{A}$ is an acyclic subcomplex in $\mathcal{K}_i$ for every $i \in \{0,\ldots,n\}$.

In order to obtain acyclic subset compatible with filtration, first the maximal acyclic subcomplex $\mathcal{A}_0$ is constructed in $\mathcal{K}_0$ as described for instance in~\cite{acc}\footnote{The idea of the procedure is as follow - first a top dimensional cell $A$ in $\mathcal{K}_0$ is chosen. Then all its neighbor elements in $\mathcal{K}_0$ are processed. Element $B \in \mathcal{K}_0$ is added to the acyclic subcomplex iff its intersection with the current acyclic subcomplex is acyclic. For fast tests tabulated configurations for cubes and simplices can be used.}. Then elements in $\mathcal{K}_1$ are processed to construct $\mathcal{A}_1$. The element $B \in \mathcal{K}_1$ is added to to the complex iff: 
\begin{enumerate}
\item The intersection with the acyclic complex constructed so far is acyclic and \label{pt1}
\item there are no elements in boundary of $B$ intersected with $\mathcal{K}_{0}$ that are not in acyclic subcomplex $\mathcal{A}_0$. \label{pt2}
\end{enumerate}
For the higher values of filtration we use analogous criterion. The point~(\ref{pt2}) for $g(B) = i$ should be then replaced with:
\begin{enumerate}
\item there are no elements in boundary of $B$ intersected with $\mathcal{K}_{j}$ for $j < i$ that are not in acyclic subcomplex $\mathcal{A}_{i-1}$.
\end{enumerate}

In this way we ensure that the closure of the final acyclic subcomplex intersected with every level of filtration is acyclic subcomplex at this level of filtration as shown in Theorem~\ref{th:acc}. This condition is necessary as explained in Figure~\ref{fig:acyclicSubcmplx}. This fact is used in the proof of Theorem~\ref{thm:acyclic}.

\begin{theorem}
\label{th:acc}
As a result of the presented procedure an acyclic subcomplex compatible with filtration is obtained. 
\end{theorem}
\begin{proof}
Due to~(\ref{pt1}) the obtained complex is acyclic. From point~(\ref{pt2}) we have that $\mathcal{K}_i \setminus \mathcal{A} = \mathcal{K}_i \setminus \mathcal{A}_i$. Therefore we obtain acyclic subcomplex compatible with filtration.
\end{proof}

\begin{figure}[!h]
\centerline{\includegraphics[scale=0.3]{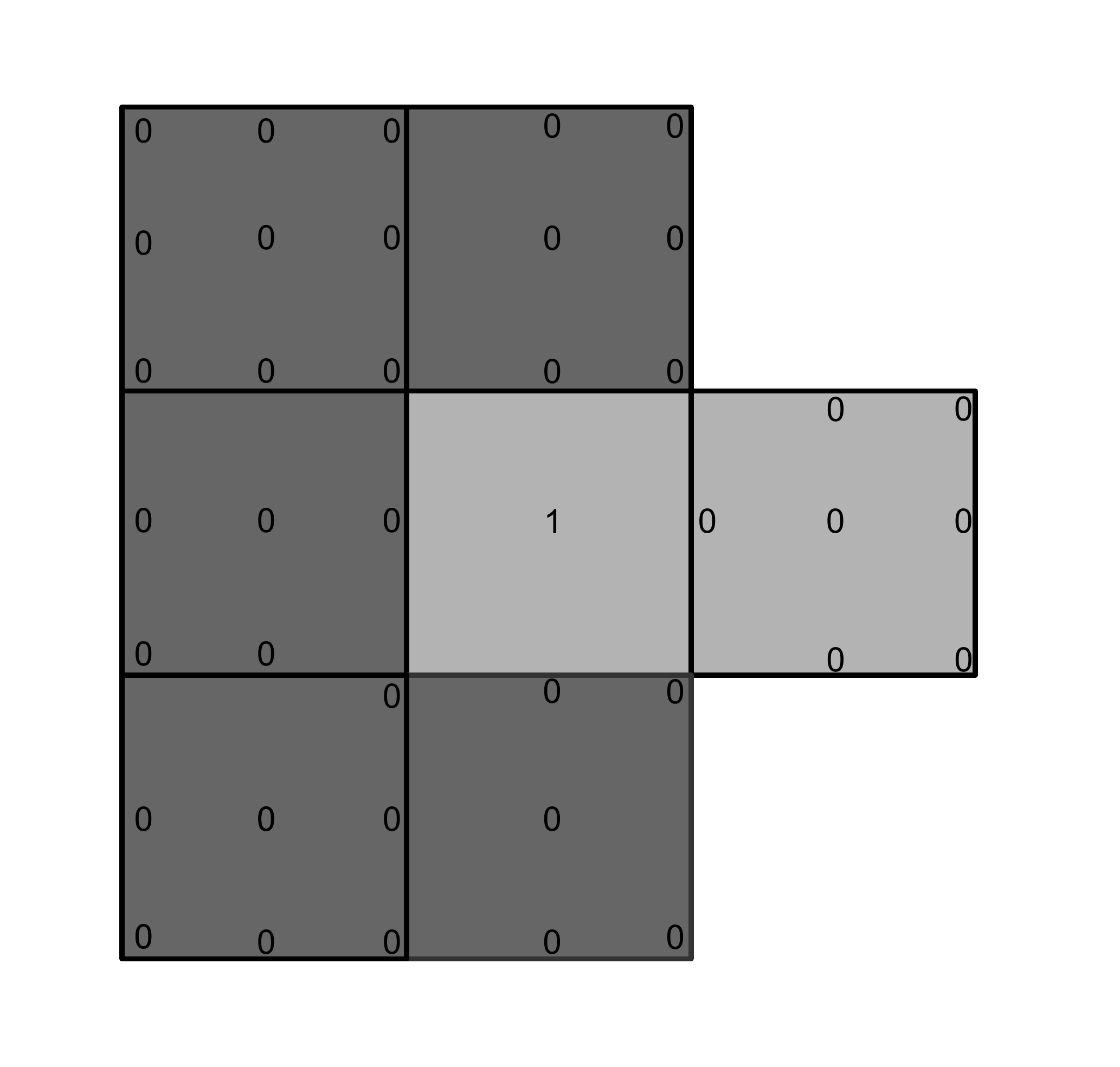}}
\caption{At the first filtration level the 2-dimensional elements with filtration $0$ (marked with darker gray) belong to the acyclic subcomplex. If, as in~\cite{acc}, we considered only intersection with acyclic subcomplex when generating the acyclic subcomplex at the filtration level $1$, the middle 2-dimensional element would be added to the acyclic subset at filtration level $1$. In the end, when the whole (closed) acyclic subspace is removed, we lose the $[0,1]$ persistence interval in dimension one. This cannot happen when the extra restrictions are imposed.
}
\label{fig:acyclicSubcmplx}
\end{figure}

Let us show that such a procedure do not change persistence in dimension greater of equal $1$.
\begin{theorem}
\label{thm:acyclic}
Let $\mathcal{K}_0 \subset \ldots \subset \mathcal{K}_n = \mathcal{K}$ be a filtered complex and let $\mathcal{A}_0 \subset \ldots \subset \mathcal{A}_n = \mathcal{A}$, such that $\mathcal{A}_i = \mathcal{K}_i \setminus \mathcal{A}$, be an acyclic subcomplex compatible with filtration. Then persistent homologies $H^p_l(\mathcal{K})$ and $H^p_l(\mathcal{K} \setminus \mathcal{A})$ are the same for $l > 0$.
\end{theorem} 
\begin{proof}
Again let us write the following diagram:
\[
\begin{CD} 
  \ldots @>H(i_{l-1})>> H_{*}(\mathcal{K}_l) @>H(i_{l})>> H_{*}(\mathcal{K}_{l+1}) @>H(i_{l+1})>> \ldots \\
  @VVV@VV H(j_{l}) V@VV H(j_{l+1}) V @VVV\\
  \ldots @>H(k_{l-1})>> H_{*}(\mathcal{K}_l \setminus \mathcal{A}_l) @>H(k_{l})>> H_{*}(\mathcal{K}_{l+1} \setminus \mathcal{A}_{l+1}) @>H(k_{l+1})>> \ldots\\
\end{CD}
\]

The horizontal arrows are induced by inclusion. 
While it is straightforward for the upper sequence, the lower one require some explanation, since in general both $\mathcal{K}_l \subset \mathcal{K}_{l+1}$ and $\mathcal{A}_l \subset \mathcal{A}_{l+1}$. The $k_l$ is an inclusion, since $\mathcal{A}_{l+1} \setminus \mathcal{A}_l \subset \mathcal{K}_{l+1} \setminus \mathcal{K}_l$.

From  the exact sequence of a pair $(\mathcal{K}_l , \mathcal{A}_l)$:
\[ \rightarrow H_n(\mathcal{A}_l) \rightarrow H_n(\mathcal{K}_l) \rightarrow H_n( \mathcal{K}_l , \mathcal{A}_l ) \rightarrow \]
since $H_n(\mathcal{A}_l) = 0$ for $n > 0$ we have that $H_n(\mathcal{K}_l)$ is isomorphic to $H_n( \mathcal{K}_l , \mathcal{A}_l )$. From Theorem~\ref{th:MMBB31} we have that $H_n(\mathcal{K}_l , \mathcal{A}_l )$ is isomorphic to $H_n( \mathcal{K}_l \setminus \mathcal{A}_l )$. Therefore all the vertical arrows are isomorphisms for $n>0$. Again, as in Theorem~\ref{th:reductions} one can show that all the squares commute since $\mathcal{A}_{l+1} \setminus \mathcal{A}_l \subset \mathcal{K}_{l+1} \setminus \mathcal{K}_l$. Consequently, from Theorem~\ref{thm:pet}, the persistence modules corresponding to lower and upper complexes are equal for $n>0$.
\end{proof}
As for the zero dimensional persistence here exactly the same situation as the one presented in Figure~\ref{fig:zeroDimIsLost} occurs. Therefore all the information about zero dimensional persistence is lost. They can be retrieved in near linear time by using find union data structure.

Algorithm~\ref{alg:acc} summarizes the acyclic subspace method for persistence. In this algorithm only the top dimensional elements need to be represented.
\begin{algorithm}[!ht]
  \small
  \caption{Acyclic subspace algorithm for persistence.}
  \label{alg:acc}
  \begin{algorithmic}
	\REQUIRE $\mathcal{K}$ - cell complex with filtration $g : \mathcal{K} \rightarrow \mathbb{Z}$;
	\ENSURE $\mathcal{A}$ - acyclic subcomplex of $\mathcal{K}$;
	\STATE $min = $ minimal value of $g$ on $\mathcal{K}$;
	\STATE $\mathcal{A} = \emptyset$, the acyclic subcomplex;
	\FOR {every connected component $\mathcal{C}$ of $\mathcal{K}$}
		\STATE Pick any top dimensional cell $T \in \mathcal{C}$ having minimal filtration value in its connected component. Set $\mathcal{A} = \mathcal{A} \cup \{T\}$;
	\ENDFOR
	\FOR {every filtration value $i$ starting from $min$ in increasing order}
		\STATE List of cells \texttt{Q} $= \emptyset$;
		\FOR {every $T$ such that $g(T) = i$}
			\IF {$T \cap \mathcal{A}$ is acyclic and for every element $a \in T \setminus (T \cap \mathcal{A})$, $g(a) = i$}
				\STATE \texttt{Q} $=$ \texttt{Q} $\cup\ T$;
			\ENDIF
		\ENDFOR
		\WHILE { \texttt{Q} is not empty }
			\STATE $T\ =\ $\texttt{Q}$.dequeue()$;
			\IF {$T \cap \mathcal{A}$ is acyclic}
				\STATE $\mathcal{A} = \mathcal{A} \cup T$;
				\FOR { every cell $T'$ incident to $T$ such that $g(T') = i$ and $T' \not \in \mathcal{A}$ }
					\IF { $T' \cap \mathcal{A}$ is acyclic and for every element $a \in T' \setminus (T' \cap \mathcal{A})$, $g(a) = i$ }
						\STATE \texttt{Q} $=$ \texttt{Q} $\cup\ T'$;
					\ENDIF
				\ENDFOR
			\ENDIF
		\ENDWHILE
	\ENDFOR
  \end{algorithmic}
\end{algorithm} 

The complexity of this algorithm is linear provided the cells of the complex and the neighboring cells can be iterated according to the filtration levels of the function $g$.

\section{Smoothing up the data}
\label{sec:smoothing}
Let us fix a complex $\mathcal{K}$ and a filtering function $f$. The number of elements reduced with described algorithms strongly depends on the filtering function $f$. It may happen, especially for noisy data that no reduction can be made, because the value of cell and some of its faces differs infinitesimally. In this appendix we show a heuristic for \emph{denoising} such data, controlling the changes of persistence. This way we increase the effectiveness of the presented reduction methods. This idea is based on stability results for persistence~\cite{stability}. As recalled in Equation~\ref{eq:stability}, stability of persistence states that under the $\epsilon$ change (in the maximum norm) of the filtering function the persistence diagrams will change by no more than $\epsilon$ in the so-called bottleneck distance~\cite{stability}. 

We start with free face collapses and coreductions. One can view the denoising procedure as constructing a perturbed function $f' : \mathcal{K} \rightarrow \mathbb{Z}$ (we assume that $f' = f$ on all cells in which the perturbation did not take place). In this case when performing free face collapsing or coreductions the requirement $g(A) = g(b)$ can be relaxed by checking:
\begin{enumerate}
\item If $f(A) - f(b) < \epsilon$ and 
\item if for every $B$ being a coface of $b$ we have $f(B) \geq f(A)$
\end{enumerate}
then the filtration value of the cell $b$ is perturbed to $f'(b) = f(A)$ and a reduction of a pair $(A,b)$ can be made when a complex $\mathcal{K}$ with a filtering function $f'$ is considered. 

Let us show that $f'$ is a sublevel-complex filtration and $\lVert f-f' \rVert_\infty = max_{a \in \mathcal{K}} | f(a) - f'(a) | \leq \epsilon$. The fact that $f'$ is a filtration of $\mathcal{K}$ follows from point (2) above and the fact that if a filtration value of a coface of $b$, such that $f(b) \neq f'(b)$, is changed then they are increased. The fact that for every $a \in \mathcal{K}$, $| f(a) - f'(a) | \leq \epsilon$ follows from the condition (1). Therefore, from stability of persistence~\cite{stability} we have $d_B(H^p(\mathcal{K},f), H^p(\mathcal{K},f')) \leq \epsilon$.

It is known from Discrete Morse theory~\cite{leviner} that maximizing the number of reduced elements is NP-complete. Therefore a greedy strategy seems to be a viable option. One can perform a greedy reduction together with changing the value of the function if the conditions presented above are satisfied.  We want to stress that once we want to change the value of a cell $a \in \mathcal{K}$ from $f(a)$ to $f'(a)$, then the inequality $f(B) \geq f'(a)$, for every $B$ coface of $a$, need to be checked in the initial complex (so we are not allowed to disregard the elements that has already been reduced). Otherwise the changes will accumulate and the $\epsilon$ precision will be lost as presented in Figure~\ref{fig:toleranceStability}.

\begin{figure}[!h]
\centerline{\includegraphics[scale=0.5]{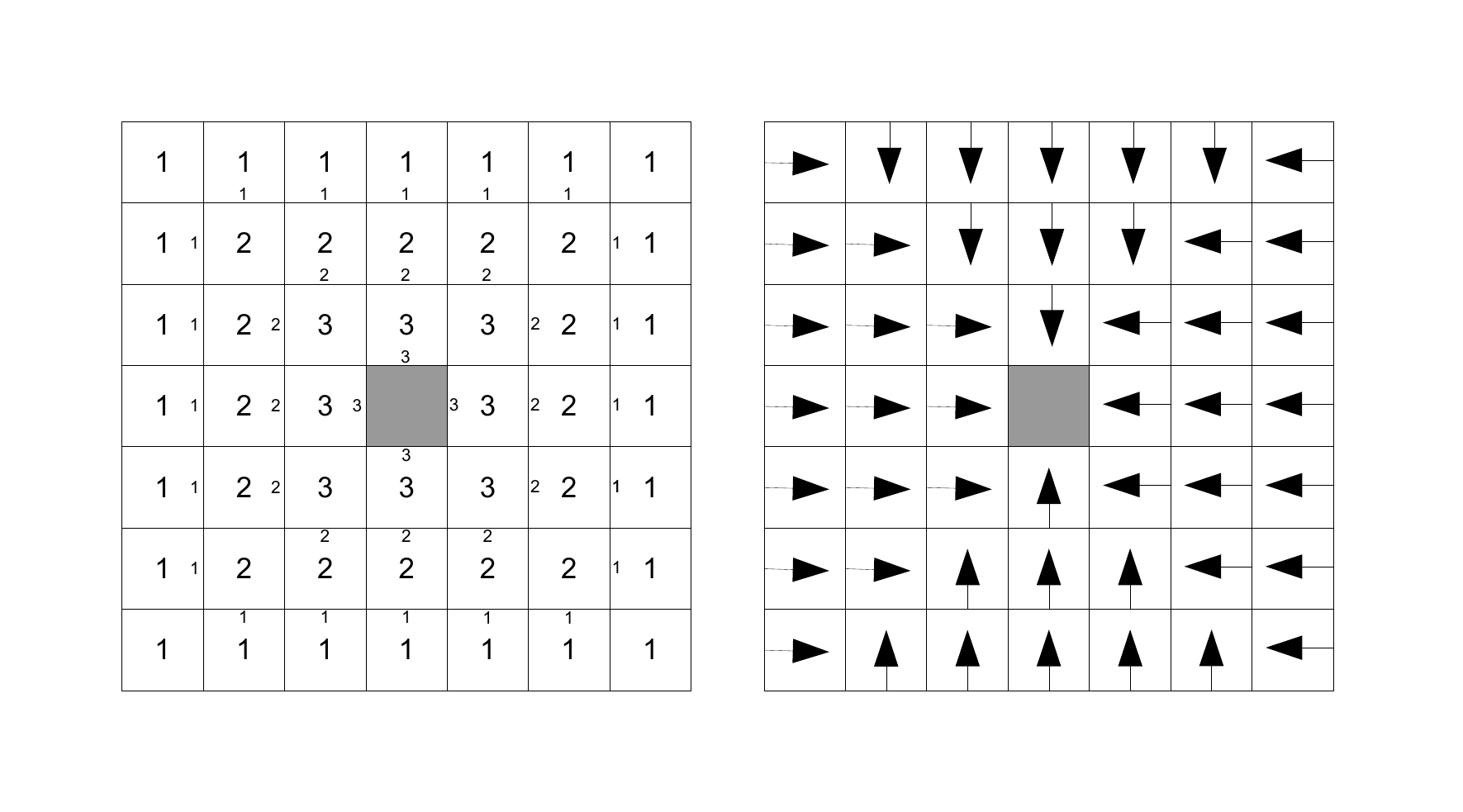}}
\caption{Example showing that when changing the filtering function values after some reductions were made, one should always check the presented conditions on the initial complex, not the reduced one. Otherwise the changes could accumulate and one would lose the bounds on the distances of the output persistence, as presented in this example. On the left the initial complex with the initial function values. Filtration is given on top dimensional cells and lower dimensional cell has the filtration equal to minimum of filtration of incident 2-dimensional cells. Suppose we want to obtain persistence with a tolerance $\epsilon = 1$. Then, if the reduced elements are forgotten and they are not taken care of in the checking $f(B) \geq f'(b)$ for $b$ being a face of $B$, then all the reductions presented on the right can be made which clearly gives an interval $[3,\infty]$ in the dimension $1$. The correct one, $[1,\infty]$ is farther than $\epsilon$ from the provided answer. By enlarging this example the difference can be made arbitrarily large.}
\label{fig:toleranceStability}
\end{figure}

Similar idea can be applied to acyclic subcomplex method for persistence. Let us assume that the filtration $f$ is given only on top dimensional cells of a complex $\mathcal{K}$. Then $f$ is assumed to be extended to other cells by using lower star filtration (i.e. every cell have a filtration equal to the minimal filtration of the incident top dimensional cells). 
Once one perturbs the values of top dimensional cells ($f'$ is the perturbed function) so that for each $A \in \mathcal{K}$ being a top dimensional cell $|f(A)-f'(A)|\leq \epsilon$, and constructed lower star filtration based on $f'$ for every $a \in \mathcal{K}$ we will have $|f(a) - f'(a)| \leq \epsilon$. Consequently $d_B(H^p(\mathcal{K},f), H^p(\mathcal{K},f')) \leq \epsilon$. Due to this property one can do two things to increase the efficiency of acyclic subcomplex method:
\begin{enumerate}
\item Reduce the number of filtration levels of $f'$ with respect to $f$ as much as possible. This allows a better complexity of the algorithm (we remind that the complexity of this algorithm is dependent to the number of filtration levels).
\item Use greedy strategy when constricting the maximal acyclic subcomplex by changing the value of neighboring cells used to construct locally larger acyclic subcomplex.
\end{enumerate}
In both cases we face a hard optimization problem. Therefore, the only feasible solution would be to use a greedy strategy. 

\section{Conclusions.}
In this paper, we adopted existing reduction techniques, which were beneficial in homology computation, to the context of persistent homology. In many practical cases such a reductions should be used as a preprocessing step to avoid storing and processing the entire boundary matrix of the initial data. As in the case of standard homology computation one can use the presented reductions in a sequence -- acyclic subspace, then elementary collapses and coreductions. 

\end{document}